\documentclass[10pt]{amsart}
\usepackage{amscd,amsmath,amssymb,amsthm,enumerate,graphicx}
\newtheorem{definition}[equation]{Definition}
\newtheorem{lemma}[equation]{Lemma}
\newtheorem{proposition}[equation]{Proposition}
\newtheorem{theorem}[equation]{Theorem}

\newtheorem{remark}[equation]{Remark}
\newtheorem{conjecture}[equation]{Conjecture}
\newtheorem{problem}[equation]{Problem}

\newcommand\lemmaref[1]{Lemma~\ref{#1}}
\newcommand\propositionref[1]{Proposition~\ref{#1}}

\newcommand\remarkref[1]{Remark~\ref{#1}}
\newcommand\conjectureref[1]{Conjecture~\ref{#1}}
\newcommand\problemref[1]{Problem~\ref{#1}}
\title{Discretely decomposable restrictions of $(\mathfrak{g},K)$-modules for Klein four symmetric pairs of exceptional Lie groups of Hermitian type}
\author{Haian HE}
\date{}
\address{Department of Mathematics,
College of Sciences, Shanghai University,
No.\ 99 Shangda Road, Baoshan District,
Shanghai, China P.\ R.\ 200444}
\email{hebe.hsinchu@yahoo.com.tw}

\begin{document}
\maketitle
\begin{abstract}
Let $(G,G^\Gamma)$ be a Klein four symmetric pair. The author wants to classify all the Klein four symmetric pairs $(G,G^\Gamma)$ such that there exists at least one nontrivial unitarizable simple $(\mathfrak{g},K)$-module $\pi_K$ that is discretely decomposable as a $(\mathfrak{g}^\Gamma,K^\Gamma)$-module. In this article, three assumptions will be made. Firstly, $G$ is an exceptional Lie group of Hermitian type, i.e., $G=\mathrm{E}_{6(-14)}$ or $\mathrm{E}_{7(-25)}$. Secondly, $G^\Gamma$ is noncompact. Thirdly, there exists an element $\sigma\in\Gamma$ corresponding to a symmetric pair of anti-holomorphic type such that $\pi_K$ is discretely decomposable as a $(\mathfrak{g}^\sigma,K^\sigma)$-module.
\end{abstract}
\section{Introduction}
Let $G$ be a noncompact simple Lie group, and $G'$ a reductive subgroup of $G$. Write $\mathfrak{g}$ and $\mathfrak{g}'$ for the complexified Lie algebras of $G$ and $G'$ respectively. Take a maximal compact subgroup $K$ of $G$ such that $K':=K\cap G'$ is a maximal compact subgroup of $G'$. An important problem in branching theory is to classify the reductive Lie group pairs $(G,G')$ such that there exists at least one nontrivial unitarizable simple $(\mathfrak{g},K)$-module that is discretely decomposable as a $(\mathfrak{g}',K')$-module.

If $G'$ is compact, this always holds obviously. Hence in this article, the author will only considers the cases when $G'$ is noncompact. If $(G,G')$ is a symmetric pair, the problem was solved completely by Toshiyuki KOBAYASHI and Yoshiki \={O}SHIMA in \cite[Corollary 5.8]{KO2}. If $G$ is simple Lie group of Hermitian type, and $G'$ contains the center of $K$, then any irreducible unitary highest weight representation $\pi$ of $G$ is $K'$-admissible \cite{Ko3}, and it follows from \cite[Proposition 1.6]{Ko4} that the underlying $(\mathfrak{g},K)$-module $\pi_K$ is discretely decomposable as a $(\mathfrak{g}',K')$-module. In particular, it is the case when $(G,G^\Gamma)$ is a Klein four symmetric pair of holomorphic type \cite[Corollary 7]{H1}, where $\Gamma$ is a Klein four subgroup of the automorphism group $\mathrm{Aut}G$ and $G^\Gamma$ is the subgroup of the fixed points under the action of $\Gamma$ on $G$. In \cite{H1} and \cite{H2}, the author classified all the Klein four symmetric pairs of holomorphic type for $\mathrm{E}_{6(-14)}$ and $\mathrm{E}_{7(-25)}$ respectively. A more general question for Klein four symmetric pair is as follow.
\begin{problem}\label{20}
Classify all the Klein four symmetric pairs $(G,G^\Gamma)$ such that there exists at least one nontrivial unitarizable simple $(\mathfrak{g},K)$-module $\pi_K$ that is discretely decomposable as a $(\mathfrak{g}^\Gamma,K^\Gamma)$-module.
\end{problem}
\problemref{20} is hard, even if one just considers the cases when $G=\mathrm{E}_{6(-14)}$ or $\mathrm{E}_{7(-25)}$. However, one may consider the following problem which seems not so hard as \problemref{20}.
\begin{problem}\label{21}
Classify all the Klein four symmetric pairs $(G,G^\Gamma)$ such that there exists at least one nontrivial unitarizable simple $(\mathfrak{g},K)$-module $\pi_K$ satisfying two conditions:
\begin{enumerate}[$\bullet$]
\item $\pi_K$ is discretely decomposable as a $(\mathfrak{g}^\Gamma,K^\Gamma)$-module;
\item $\pi_K$ is discretely decomposable as a $(\mathfrak{g}^\sigma,K^\sigma)$-module for some involutive automorphism $\sigma\in\Gamma$ of anti-holomorphic type.
\end{enumerate}
\end{problem}
where $\mathfrak{g}^\sigma$ (respectively, $K^\sigma$) is the subalgebra (respectively, subgroup) of the fixed points under the action of $\sigma$ on $\mathfrak{g}$ (respectively, $K^\sigma$).

The article is organized as follows. After recalling the definitions and the fundamental results of discrete decomposability and Klein four symmetric pairs in Section 2, the author will make use of the classification result in \cite[Theorem 5.2 \& Table 1]{KO2} to show that the only pair worthy to be studied is $(\mathrm{E}_{6(-14)},\mathrm{Spin}(8,1))$ in Section 3. Then the author will show that $(G,G^\Gamma)=(\mathrm{E}_{6(-14)},\mathrm{Spin}(8,1))$ is indeed a Klein four symmetric pair, and there exists a nontrivial unitarizable simple $(\mathfrak{g},K)$-module $\pi_K$ satisfying the two conditions in \problemref{21}. This involves some computations, and the main theorem will be displayed in Section 4. Finally, in Section 5, the author will briefly discuss $G'$-admissibility of unitary irreducible representations of $G$, based on a conjecture by Toshiyuki KOBAYASHI in \cite{Ko1}.

When writing this article, the author communicated with professor Jorge VARGAS frequently, who gave some important suggestions. The author would like to express the thankfulness to professor Jorge VARGAS.
\section{Preliminary}
\subsection{Discrete decomposability}
Let $G$ be a noncompact simple Lie group with the complexified Lie algebra $\mathfrak{g}$, and $G'$ a reductive subgroup of $G$ with the complexified Lie algebra $\mathfrak{g}'$. Take a maximal compact subgroup $K$ of $G$, which is the subgroup of the fixed points under the action of a Cartan involution $\theta$ on $G$. In particular, it is assumed that $\theta(G')=G'$; equivalently, $K':=K\cap G'$ is a maximal compact subgroup of $G'$.
\begin{definition}\label{1}
A $(\mathfrak{g},K)$-module $X$ is called discretely decomposable as a $(\mathfrak{g}',K')$-module if there exists an increasing filtration $\{X_i\}_{i\in\mathbb{Z}^+}$ of $(\mathfrak{g}',K')$-modules such that $\bigcup\limits_{i\in\mathbb{Z}^+}X_i=X$ and $X_i$ is of finite length as a $(\mathfrak{g}',K')$-module for any $i\in\mathbb{Z}^+$.
\end{definition}
\begin{remark}\label{2}
Suppose that $X$ is a unitarizable $(\mathfrak{g},K)$-module. According to \cite[Lemma 1.3]{Ko4}, $X$ is discretely decomposable as a $(\mathfrak{g}',K')$-module if and only if $X$ is isomorphic to an algebraic direct sum of simple $(\mathfrak{g}',K')$-modules.
\end{remark}
\begin{proposition}\label{3}
Let $\pi_K$ be a unitarizable simple $(\mathfrak{g},K)$-module. Then $\pi_K$ is discretely decomposable as a $(\mathfrak{g}',K')$-module if and only if there exists a simple $(\mathfrak{g}',K')$-module $\tau_{K'}$ such that $\mathrm{Hom}_{\mathfrak{g}',K'}(\tau_{K'},\pi_K)\neq\{0\}$.
\end{proposition}
\begin{proof}
See \cite[Lemma 1.5]{Ko4}.
\end{proof}
\subsection{Klein four symmetric pairs}
\begin{definition}\label{4}
Let $G$ (respectively, $\mathfrak{g}_0$) be a real simple Lie group (respectively, Lie algebra), and let $\Gamma$ be a Klein four subgroup of $\mathrm{Aut}G$ (respectively, $\mathrm{Aut}\mathfrak{g}_0$). Denote by $G^\Gamma$ (respectively, $\mathfrak{g}_0^\Gamma$) the subgroup (respectively, subalgebra) of the fixed points under the action of all elements in $\Gamma$ on $G$ (respectively, $\mathfrak{g}_0$). Then $(G,G^\Gamma)$ (respectively, $(\mathfrak{g}_0,\mathfrak{g}_0^\Gamma)$) is called a Klein four symmetric pair.
\end{definition}
\begin{remark}\label{5}
In particular, if $G$ (respectively, $\mathfrak{g}_0$) is a real simple Lie group (respectively, Lie algebra) of Hermitian type, and every nonidentity element $\sigma\in\Gamma$ defines a symmetric pair of holomorphic type, then $(G,G^\Gamma)$ (respectively, $(\mathfrak{g}_0,\mathfrak{g}_0^\Gamma)$) is called a Klein four symmetric pair of holomorphic type. In this article, the author considers Klein four symmetric pairs of non-holomorphic type; namely, there exists at least one element $\sigma\in\Gamma$ which defines a symmetric pair of anti-holomorphic type.
\end{remark}
With the notations above, for a Klein four symmetric pair $(G,G^\Gamma)$, the maximal compact subgroup $K$ of $G$ is always supposed to satisfy $\sigma(K)=K$ for all $\sigma\in\Gamma$, so that $K^\Gamma=K\cap G^\Gamma$ is a maximal compact subgroup of $G^\Gamma$.
\begin{proposition}\label{6}
If $(G,G^\Gamma)$ is a Klein four symmetric pair of holomorphic type, then any irreducible unitary highest weight representation $\pi$ of $G$ is $K^\Gamma$-admissible, and hence the underlying $(\mathfrak{g},K)$-module $\pi_K$ is discretely decomposable as a $(\mathfrak{g}^\Gamma,K^\Gamma)$-module.
\end{proposition}
\begin{proof}
See \cite[Theorem 2.9(1) \& Example 2.13 \& Example 3.3]{Ko3} or \cite[Corollary 7]{H1} for the first statement, and the second statement follows from \cite[Proposition 1.6]{Ko4}.
\end{proof}
\section{An investigation to classfication}
Let $G$ be a noncompact simple Lie group, and $\Gamma$ a Klein four subgroup of $\mathrm{Aut}G$. Then $(G,G^\Gamma)$ is a Klein four symmetric pair and $(G,G^\sigma)$ is a symmetric pair for each nonidentity $\sigma\in\Gamma$. Take a maximal compact subgroup $K$ of $G$ such that $\sigma(K)=K$ for all $\sigma\in\Gamma$.
\begin{lemma}\label{9}
Let $\pi_K$ be a unitarizable simple $(\mathfrak{g},K)$-module, and let $\sigma\in\Gamma$. If $\pi_K$ is discretely decomposable as a $(\mathfrak{g}^\Gamma,K^\Gamma)$-module and is also discretely decomposable as a $(\mathfrak{g}^\sigma,K^\sigma)$-module, then there exists a unitarizable simple $(\mathfrak{g}^\sigma,K^\sigma)$-module which is discretely decomposable as a $(\mathfrak{g}^\Gamma,K^\Gamma)$-module.
\end{lemma}
\begin{proof}
By the assumption that $\pi_K$ is discrete decomposable as a $(\mathfrak{g}^\sigma,K^\sigma)$-module, it follows from \remarkref{2} that $\pi_K=\bigoplus\limits_im(i)\eta_{iK^\sigma}$ is an algebraic direct sum of simple $(\mathfrak{g}^\sigma,K^\sigma)$-modules because $\pi_K$ is unitarizable, where $m(i)$ is the multiplicity of $\eta_{iK^\sigma}$. On the other hand, $\pi_K$ is discretely decomposable as a $(\mathfrak{g}^\Gamma,K^\Gamma)$-module, so according to \propositionref{3}, there exists a simple $(\mathfrak{g}^\Gamma,K^\Gamma)$-module $\tau_{K^\Gamma}$ such that $\mathrm{Hom}_{\mathfrak{g}^\Gamma,K^\Gamma}(\tau_{K^\Gamma},\pi_K)\neq\{0\}$. Hence, one has $\{0\}\neq\mathrm{Hom}_{\mathfrak{g}^\Gamma,K^\Gamma}(\tau_{K^\Gamma},\bigoplus\limits_im(i)\eta_{iK^\sigma})\subseteq \prod\limits_im(i)\mathrm{Hom}_{\mathfrak{g}^\Gamma,K^\Gamma}(\tau_{K^\Gamma},\eta_{iK^\sigma})$, and there is at least one simple $(\mathfrak{g}^\sigma,K^\sigma)$-module $\eta_{iK^\sigma}$ such that $\mathrm{Hom}_{\mathfrak{g}^\Gamma,K^\Gamma}(\tau_{K^\Gamma},\eta_{iK^\sigma})\neq0$. Again by \propositionref{3}, $\eta_{iK^\sigma}$ is discretely decomposable as a $(\mathfrak{g}^\Gamma,K^\Gamma)$-module, and it is unitarizable because $\pi_K$ is unitarizable.
\end{proof}
The author will find out all the Klein four symmetric pairs $(G,G^\Gamma)$ of exceptional Lie groups of Hermitian type as required in \problemref{21}. Since the author excludes the cases when $G^\Gamma$ is compact, $\Gamma$ is not supposed to contain any Cartan involution of $G$. On the other hand, it is known from \cite{HM} that, for any nontrivial irreducible unitary representation of $G$, its restriction to its reductive subgroup $G'$ contains no trivial subrepresentations. Hence in \lemmaref{9}, if $\pi_K$ is supposed to be nontrivial, all the direct summands are nontrivial. Because of the conditions in \problemref{21}, \propositionref{6}, and \lemmaref{9}, the author needs to find reductive Lie group triple $(G,G^\sigma,G^\Gamma)$ such that
\begin{enumerate}[$\bullet$]
\item $G$ is either $\mathrm{E}_{6(-14)}$ or $\mathrm{E}_{7(-25)}$.
\item both $G^\sigma$ and $G^\Gamma$ are noncompact.
\item $(G,G^\sigma)$ is a symmetric pair of anti-holomorphic type, such that there exists at least one nontrivial unitarizable simple $(\mathfrak{g},K)$-module which is discretely decomposable as a $(\mathfrak{g}^\sigma,K^\sigma)$-module.
\item $(G^\sigma,G^\Gamma)$ is a symmetric pair, such that there exists at least one nontrivial unitarizable simple $(\mathfrak{g}^\sigma,K^\sigma)$-module which is discretely decomposable as a $(\mathfrak{g}^\Gamma,K^\Gamma)$-module.
\end{enumerate}
According to the classification result \cite[Theorem 5.2 \& Table 1]{KO2}, $(\mathrm{E}_{6(-14)},\mathrm{F}_{4(-20)},\mathrm{Spin}(8,1))$ (up to covering group) is the only triple that satisfies the above four conditions.

In the next section, the author will show that $(\mathfrak{g}_0,\mathfrak{g}_0^\Gamma)=(\mathfrak{e}_{6(-14)},\mathfrak{so}(8,1))$ is a Klein four symmetric pair. Moreover, $(G,G^\Gamma)$ satisfies the conditions in \problemref{21}.
\section{$(\mathfrak{e}_{6(-14)},\mathfrak{so}(8,1))$}
\subsection{Elementary abelian 2-subgroups in $\mathrm{Aut}\mathfrak{e}_{6(-78)}$}
Recall the construction of the elementary abelian 2-subgroups in the compact Lie group $\mathrm{Aut}\mathfrak{e}_{6(-78)}$ in \cite{Y}. Let $\mathfrak{e}_6$ be the complex simple Lie algebra of type $\mathrm{E}_6$. Fix a Cartan subalgebra of $\mathfrak{e}_6$ and a simple root system $\{\alpha_i\mid1\leq i\leq6\}$, the Dynkin diagram of which is given in Figure 1.
\begin{figure}
\centering \scalebox{0.7}{\includegraphics{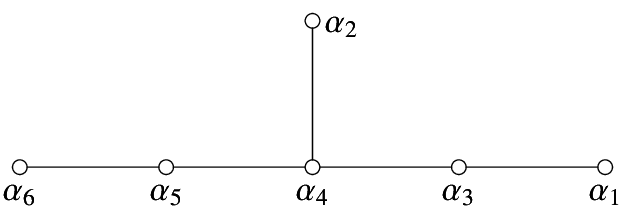}}
\caption{Dynkin diagram of $\mathrm{E}_6$.}
\end{figure}
For each root $\alpha$, denote by $H_\alpha$ its coroot, and denote by $X_\alpha$ the normalized root vector so that $[X_\alpha,X_{-\alpha}]=H_\alpha$. Moreover, one can normalize $X_\alpha$ appropriately such that\[\mathrm{Span}_\mathbb{R}\{X_\alpha-X_{-\alpha},\sqrt{-1}(X_\alpha+X_{-\alpha}),\sqrt{-1}H_\alpha\mid\alpha:\textrm{positive root}\}\cong\mathfrak{e}_{6(-78)}\]is a compact real form of $\mathfrak{g}$ by \cite{Kn}. It is well known that $\mathrm{Aut}\mathfrak{e}_{6(-78)}/\mathrm{Int}\mathfrak{e}_{6(-78)}\cong\mathrm{Aut}\mathfrak{e}_6/\mathrm{Int}\mathfrak{e}_6$, which is just the automorphism group of the Dynkin diagram.

Follow the constructions of involutive automorphisms of $\mathfrak{e}_{6(-78)}$ in \cite{HY}. Let $\omega$ be the specific involutive automorphism of the Dynkin diagram defined by
\begin{eqnarray*}
\begin{array}{rclcrcl}
\omega(H_{\alpha_1})=H_{\alpha_6},&&\omega(X_{\pm\alpha_1})=X_{\pm\alpha_6},\\
\omega(H_{\alpha_2})=H_{\alpha_2},&&\omega(X_{\pm\alpha_2})=X_{\pm\alpha_2},\\
\omega(H_{\alpha_3})=H_{\alpha_5},&&\omega(X_{\pm\alpha_3})=X_{\pm\alpha_5},\\
\omega(H_{\alpha_4})=H_{\alpha_4},&&\omega(X_{\pm\alpha_4})=X_{\pm\alpha_4}.
\end{array}
\end{eqnarray*}
Let $\sigma_1=\mathrm{exp}(\sqrt{-1}\pi H_{\alpha_2})$, $\sigma_2=\mathrm{exp}(\sqrt{-1}\pi(H_{\alpha_1}+H_{\alpha_6}))$, $\sigma_3=\omega$, $\sigma_4=\omega\mathrm{exp}\\(\sqrt{-1}\pi H_{\alpha_2})$, where $\mathrm{exp}$ represents the exponential map from $\mathfrak{e}_{6(-78)}$ to $\mathrm{Aut}\mathfrak{e}_{6(-78)}$. Then $\sigma_1$, $\sigma_2$, $\sigma_3$, and $\sigma_4$ represent all conjugacy classes of involutions in $\mathrm{Aut}\mathfrak{e}_{6(-78)}$, which correspond to real forms $\mathfrak{e}_{6(2)}$, $\mathfrak{e}_{6(-14)}$, $\mathfrak{e}_{6(-26)}$, and $\mathfrak{e}_{6(6)}$.

From \cite{HY}, it is known that $(\mathrm{Int}\mathfrak{e}_{6(-78)})^{\sigma_3}\cong\mathrm{F}_{4(-52)}$, the compact Lie group of type $\mathrm{F}_4$, and there exist involutive automorphisms $\tau_1$ and $\tau_2$ of $\mathrm{F}_{4(-52)}$ such that $\mathfrak{f}_{4(-52)}^{\tau_1}\cong\mathfrak{sp}(3)\oplus\mathfrak{sp}(1)$ and $\mathfrak{f}_{4(-52)}^{\tau_2}\cong\mathfrak{so}(9)$, where $\mathfrak{f}_{4(-52)}$ denotes the compact Lie algebra of type $\mathrm{F}_4$. Moreover, $\tau_1$ and $\tau_2$ are conjugate to $\sigma_1$ and $\sigma_2$ in $\mathrm{Aut}\mathfrak{e}_{6(-78)}$ respectively, which represent all conjugacy classes of involutions in $\mathrm{Aut}\mathfrak{f}_{4(-52)}$.

Now one has $((\mathrm{Int}\mathfrak{e}_{6(-78)})^{\sigma_3})^{\tau_1}\cong\mathrm{Sp}(3)\times\mathrm{Sp}(1)/\langle(-I_3,-1)\rangle$, where $I_3$ is the $3\times3$ identity matrix. Let $\mathbf{i}$, $\mathbf{j}$, and $\mathbf{k}$ denote the fundamental quaternion units, and then set $x_0=\sigma_3$, $x_1=\tau_1=(I_3,-1)$, $x_2=(\mathbf{i}I_3,\mathbf{i})$, $x_3=(\mathbf{j}I_3,\mathbf{j})$, $x_4=(\left(\begin{array}{ccc}-1&0&0\\0&-1&0\\0&0&1\end{array}\right),1)$, and $x_5=(\left(\begin{array}{ccc}-1&0&0\\0&1&0\\0&0&-1\end{array}\right),1)$. For a pair $(r,s)$ of integers with $r\leq2$ and $s\leq3$, define\[F_{r,s}:=\langle x_0,x_1,\cdots,x_s,x_4,x_5,\cdots,x_{r+3}\rangle\]the group generated by the elements in the angle blanket. According to \cite[Proposition 6.3]{Y}, each elementary abelian 2-subgroup in $\mathrm{Aut}\mathfrak{e}_{6(-78)}$, which contains an element conjugate to $\sigma_3$, is conjugate to one of the groups in $F_{r,s}$.
\subsection{Klein four symmetric pair $(\mathfrak{e}_{6(-14)},\mathfrak{so}(8,1))$}
If $\sigma$ is an involutive automorphism of $\mathfrak{e}_{6(-78)}$, which is conjugate to $\sigma_2$ as defined above, then by the construction, $\mathfrak{e}_{6(-14)}\cong\mathfrak{e}_{6(-78)}^\sigma+\sqrt{-1}\mathfrak{e}_{6(-78)}^{-\sigma}$, where $\mathfrak{e}_{6(-78)}^{\pm\sigma}$ are the eigenspaces of eigenvalues $\pm1$ under the action of $\sigma$ on $\mathfrak{e}_{6(-78)}$. Thus, if $\langle\sigma,\eta_1,\eta_2,\cdots,\eta_n\rangle$ is an elementary abelian 2-subgroups of rank $n+1$ of $\mathrm{Aut}\mathfrak{e}_{6(-78)}$, then $\langle\eta_1,\eta_2,\cdots,\eta_n\rangle$ is an elementary abelian 2-subgroups of rank $n$ of $\mathrm{Aut}\mathfrak{e}_{6(-14)}$.

Recall the construction of $F_{r,s}$ above. If taking $r=2$ and $s=0$, then $F_{2,0}=\{x_0,x_4,x_5\}$ is an elementary abelian 2-subgroups of rank $3$ of $\mathrm{Aut}\mathfrak{e}_{6(-78)}$. It is already known that $x_0=\sigma_3$ by the construction. Moreover, by \cite[Lemma 8]{H1}, it is known that $x_4$, $x_5$, and $x_4x_5$ are the only three elements conjugate to $\sigma_2$. Thus, $\langle x_0,x_4\rangle$ is a Klein four subgroup of  $\mathrm{Aut}\mathfrak{e}_{6(-14)}\cong\mathrm{Aut}(\mathfrak{e}_{6(-78)}^{x_5}+\sqrt{-1}\mathfrak{e}_{6(-78)}^{-x_5})$ as well as that of $\mathrm{Aut}\mathfrak{e}_{6(-78)}$.
\begin{proposition}\label{10}
The pair $(\mathfrak{e}_{6(-14)},\mathfrak{so}(8,1))$ is a Klein four symmetric pair.
\end{proposition}
\begin{proof}
The author shows that $\mathfrak{e}_{6(-14)}^{\langle x_0,x_4\rangle}\cong\mathfrak{so}(8,1)$. In fact, according to the classification of the Klein four symmetric pairs in \cite[Table 4]{HY}, there is only one Klein four subgroup $\Gamma_7$ of $\mathrm{Aut}\mathfrak{e}_{6(-78)}$, which contains two elements conjugate to $\sigma_2$ and $\sigma_3$ respectively, and $\mathfrak{e}_{6(-78)}^{\Gamma_7}\cong\mathfrak{so}(9)$. Therefore, $\langle x_0,x_4\rangle\cong\Gamma_7$, and $\mathfrak{e}_{6(-14)}^{\langle x_0,x_4\rangle}$ must be some noncompact dual of $\mathfrak{so}(9)$. Moreover, $\mathfrak{e}_{6(-14)}^{\langle x_0,x_4,x_5\rangle}=\mathfrak{e}_{6(-78)}^{\langle x_0,x_4,x_5\rangle}$ is a maximal compact subalgebra of $\mathfrak{e}_{6(-14)}^{\langle x_0,x_4\rangle}$. Notice that $\mathfrak{e}_{6(-78)}^{x_0}\cong\mathfrak{f}_{4(-52)}$. Since $x_4$, $x_5$, and $x_4x_5$ are all conjugate to $\tau_2$ in $\mathrm{Aut}\mathfrak{f}_{4(-52)}$, again by \cite[Table 4]{HY}, $\mathfrak{e}_{6(-14)}^{\langle x_0,x_4,x_5\rangle}=\mathfrak{e}_{6(-78)}^{\langle x_0,x_4,x_5\rangle}\cong\mathfrak{f}_{4(-52)}^{\langle x_0,x_4\rangle}\cong\mathfrak{so}(8)$. Now $\mathfrak{e}_{6(-14)}^{\langle x_0,x_4\rangle}$ is the noncompact dual of $\mathfrak{so}(9)$ with the maximal compact subalgebra $\mathfrak{so}(8)$, and hence it is just $\mathfrak{so}(8,1)$.
\end{proof}
\begin{remark}\label{11}
The proof of \lemmaref{10} greatly depends on the classification of Klein four symmetric pairs in \cite[Table 4]{HY}, where Jingsong HUANG and Jun YU used symbol $\sigma_i$ to represent conjugacy classes for all Lie algebras. Here, the author uses $\tau_i$ for $\mathfrak{f}_{4(-52)}$ in order to distinguish the conjugacy classes between $\mathfrak{f}_{4(-52)}$ and $\mathfrak{e}_{6(-78)}$.
\end{remark}
\subsection{Discretely decomposable restriction}
For convenience, write $G=\mathrm{E}_{6(-14)}$ and denote by its complexified Lie algebra $\mathfrak{g}$. Retain the notations as above, and $\langle x_0,x_4\rangle$ is a Klein four subgroup of $\mathrm{Aut}G$ such that $G^{\langle x_0,x_4\rangle}$ is the reductive subgroup of $G$ corresponding to the subalgebra $\mathfrak{so}(8,1)$. Take a maximal compact subgroup $K$ satisfying $x_0(K)=x_4(K)=K$, and denote by $\theta$ the Cartan involution of $G$ corresponding to $K$. Since $G$ is a simple Lie group of Hermitian type with equal rank, namely, the rank of $\mathrm{E}_{6(-14)}$ coincides with the rank of its maximal compact subgroup, one may take a highest weight discrete series representation of $\mathrm{E}_{6(-14)}$.
\begin{lemma}\label{12}
Let $\mathfrak{g}_0=\mathfrak{e}_{6(-14)}$, the Lie algebra of $G=\mathrm{E}_{6(-14)}$. Then $\mathfrak{g}_0^{x_4}\cong\mathfrak{so}(8,2)\oplus\sqrt{-1}\mathbb{R}$.
\end{lemma}
\begin{proof}
It is known that $x_4$, $x_5$, and $x_4x_5$ are all conjugate to $\sigma_2$. By \cite[Table 1]{HY}, $\mathfrak{e}_{6(-78)}^{x_4}\cong\mathfrak{so}(10)\oplus\sqrt{-1}\mathbb{R}$ which is the compact dual of $\mathfrak{g}_0^{x_4}$. Moreover, by \cite[Table 4]{HY}, $\mathfrak{e}_{6(-78)}^{\langle x_4,x_5\rangle}\cong\mathfrak{so}(8)\oplus2(\sqrt{-1}\mathbb{R})$ which is a maximal compact subalgebra of $\mathfrak{g}_0^{x_4}$. Therefore, $\mathfrak{g}_0^{x_4}\cong\mathfrak{so}(8,2)\oplus\sqrt{-1}\mathbb{R}$.
\end{proof}
Choose a Cartan subalgebra $\mathfrak{t}$ of $\mathfrak{k}$, the complexification of the Lie algebra $\mathfrak{k}_0$ of $K$, and a positive system $\Delta^+(\mathfrak{k},\mathfrak{t})$. For a dominant integral weight $\lambda\in\mathfrak{t}^*$, denote by $F(\lambda)$ the irreducible representation of $K$ with the highest weight $\lambda$. Moreover, since $\mathfrak{g}_0=\mathfrak{e}_{6(-14)}$ is of Hermitian type, one has a decomposition $\mathfrak{g}=\mathfrak{k}+\mathfrak{p}_++\mathfrak{p}_-$ as a $\mathfrak{k}$-module. Let $\mathfrak{p}_-$ act as zero on $F(\lambda)$ and write $L(\lambda)$ for the unique simple quotient of the $(\mathfrak{g},K)$-module $U(\mathfrak{g})\otimes_{U(\mathfrak{k}+\mathfrak{p}_-)}F(\lambda)$, which is a lowest weight simple $(\mathfrak{g},K)$-module.

Recall the settings in Section 4.1. For each simple root $\alpha_i$, denote by $\omega_i$ the fundamental weight corresponding to $\alpha_i$. Suppose that $\alpha_6$ is the noncompact simple root corresponding to $\mathfrak{g}_0=\mathfrak{e}_{6(-14)}$. It is known from \cite{MO} that $L(3\omega_6)$ is unitarizable.
\begin{proposition}\label{15}
The unitarizable lowest weight simple $(\mathfrak{g},K)$-module $L(3\omega_6)$ is discretely decomposable as a $(\mathfrak{g}^{\langle x_0,x_4\rangle},K^{\langle x_0,x_4\rangle})(=(\mathfrak{so}(9,\mathbb{C}),\mathrm{Spin}(8)))$-module.
\end{proposition}
\begin{proof}
It is known from \cite[Setting 2.6 \& Theorem 3.1 \& Equation (3.22)]{MO} immediately that $L(3\omega_6)$ is discretely decomposable as a $(\mathfrak{g}^{x_4},K^{x_4})$-module, where $\mathfrak{g}^{x_4}\cong\mathfrak{so}(10,\mathbb{C})\oplus\mathfrak{so}(2,\mathbb{C})$ and $K^{x_4}\cong\mathrm{Spin}(8)\times\mathrm{Spin}(2)^2$ up to finite quotient. Moreover, one has a semisimple decomposition $L(3\omega_6)=\bigoplus\limits_{k=0}^{+\infty}L'(3\mu_1+k\mu_5)\boxtimes\mathbb{C}_{k+2}$ as a $(\mathfrak{g}^{x_4},K^{x_4})$-module, where $\mu_1$ and $\mu_5$ are the fundamental weights corresponding to the simple roots for $\mathfrak{so}(10,\mathbb{C})$ as in \cite[Setting 2.2]{MO}, $L'(3\mu_1+k\mu_5)$ is the lowest weight simple $(\mathfrak{g}^{x_4},K^{x_4})$-module with the parameter $3\mu_1+k\mu_5$, and $\mathbb{C}_{k+2}$ is some one dimensional $(\mathfrak{so}(2,\mathbb{C}),\mathrm{Spin}(2))$-module.

In particular, when $k=0$, $L'(3\mu_1)\boxtimes\mathbb{C}_2$ is a simple $(\mathfrak{g}^{x_4},K^{x_4})$-submodule of $L(3\omega_6)$. Also, $L'(3\mu_1)\boxtimes\mathbb{C}_2\cong L'(3\mu_1)$ as a $(\mathfrak{so}(10,\mathbb{C}),\mathrm{Spin}(8)\times\mathrm{Spin}(2))$-module. Moreover, according to \cite[Setting 2.2 \& Theorem 6.1]{MO}, $L'(3\mu_1)$ is simple as a $(\mathfrak{so}(9,\mathbb{C}),\mathrm{Spin}(8))$-module. Thus $\mathrm{Hom}_{\mathfrak{so}(9,\mathbb{C}),\mathrm{Spin}(8)}(L'(3\mu_1),L(3\omega_6))= \mathrm{Hom}_{\mathfrak{so}(9,\mathbb{C}),\mathrm{Spin}(8)}(L'(3\mu_1), \bigoplus\limits_{k=0}^{+\infty}L'(3\mu_1+k\mu_5)\boxtimes\mathbb{C}_{k+2})= \mathrm{Hom}_{\mathfrak{so}(9,\mathbb{C}),\mathrm{Spin}(8)}(L'(3\mu_1),\bigoplus\limits_{k=0}^{+\infty}L'(3\mu_1+k\mu_5))\neq\{0\}$. The conclusion follows from \propositionref{3}.
\end{proof}
\begin{lemma}\label{22}
Let $\mathfrak{g}_0=\mathfrak{e}_{6(-14)}$, the Lie algebra of $G=\mathrm{E}_{6(-14)}$. Then $\mathfrak{g}_0^{x_0}\cong\mathfrak{f}_{4(-20)}$.
\end{lemma}
\begin{proof}
The proof is similar to that for \lemmaref{12}. By \cite[Table 1]{HY}, $\mathfrak{e}_{6(-78)}^{x_0}\cong\mathfrak{f}_{4(-52)}$ which is the compact dual of $\mathfrak{g}_0^{x_0}$. Moreover, by the proof for \propositionref{10} and \cite[Table 4]{HY}, $\langle x_0,x_5\rangle\cong\Gamma_7$ and $\mathfrak{e}_{6(-78)}^{\langle x_0,x_5\rangle}\cong\mathfrak{so}(9)$ which is a maximal compact subalgebra of $\mathfrak{g}_0^{x_0}$. Therefore, $\mathfrak{g}_0^{x_0}\cong\mathfrak{f}_{4(-20)}$.
\end{proof}
\begin{proposition}\label{23}
The symmetric pair $(G,G^{x_0})=(\mathrm{E}_{6(-14)},\mathrm{F}_{4(-20)})$ is of anti-holomorphic type, and $L(3\omega_6)$ is discretely decomposable as a $(\mathfrak{g}^{x_0},K^{x_0})(=(\mathfrak{f}_{4,\mathbb{C}},\mathrm{Spin}(9)))$-module.
\end{proposition}
\begin{proof}
See \cite[Table C.2]{KO1} for the first statement. As for the second statement, it is known from \cite{MO} that $L(3\omega_6)$ is a minimal representation; in particular, it has the smallest Gelfand-Kirillov dimension. Thus, by \cite[Proposition 4.10 \& Theorem 5.2 \& Table 1]{KO2}, $L(3\omega_6)$ is discretely decomposable as a $(\mathfrak{f}_{4,\mathbb{C}},\mathrm{Spin}(9))$-module.
\end{proof}
The author ends up the section with a conclusion.
\begin{theorem}\label{16}
Let $G$ be an exceptional Lie group of Hermitian type, i.e., $G=\mathrm{E}_{6(-14)}$ or $\mathrm{E}_{7(-25)}$, and $(G,G^\Gamma)$ be a Klein four symmetric pair defined by a Klein four subgroup $\Gamma\subseteq\mathrm{Aut}G$ such that $G^\Gamma$ is noncompact. Take a maximal compact subgroup $K$ of $G$, which is stable under the action of all elements in $\Gamma$ on $G$. Then there exists a nontrivial unitarizable simple $(\mathfrak{g},K)$-module $\pi_K$ which is both discretely decomposable as a $(\mathfrak{g}^\Gamma,K^\Gamma)$-module and is discretely decomposable as a $(\mathfrak{g}^\sigma,K^\sigma)$-module for some $\sigma\in\Gamma$ of anti-holomorphic type, if and only if $(\mathfrak{g},\mathfrak{g}^\Gamma)=(\mathfrak{e}_{6(-14)},\mathfrak{so}(8,1))$.
\begin{proof}
Suppose that $(\mathfrak{g},\mathfrak{g}^\Gamma)=(\mathfrak{e}_{6(-14)},\mathfrak{so}(8,1))$. Take $\sigma\in\Gamma$ such that $\mathfrak{g}^\sigma=\mathfrak{f}_{4(-20)}$. Then $(\mathfrak{g},\mathfrak{g}^\sigma)$ is a symmetric pair of anti-holomorphic type, and $L(3\omega_6)$ is discretely decomposable as a $(\mathfrak{g}^\sigma,K^\sigma)$-module by \propositionref{23}. Moreover, it follows from \propositionref{15} that $L(3\omega_6)$ is discretely decomposable as a $(\mathfrak{g}^\Gamma,K^\Gamma)$-module. Conversely, suppose that $G^\Gamma$ is noncompact and $(G,G^\Gamma)$ is not of holomorphic type. If there exists a nontrivial unitarizable simple $(\mathfrak{g},K)$-module $\pi_K$ which is both discretely decomposable as a $(\mathfrak{g}^\Gamma,K^\Gamma)$-module and is discretely decomposable as a $(\mathfrak{g}^\sigma,K^\sigma)$-module for some $\sigma\in\Gamma$ of anti-holomorphic type, then only the pair $(G,G^\Gamma)$ corresponding to $(\mathfrak{g},\mathfrak{g}^\Gamma)=(\mathfrak{e}_{6(-14)},\mathfrak{so}(8,1))$ satisfies \lemmaref{9}.
\end{proof}
\end{theorem}
\section{A conjecture for $G^\Gamma$-admissibility}
Let $G$ be a noncompact simple Lie group, and $G'$ a reductive subgroup of $G$. Retain all the settings as in Section 1. Suppose that $\pi$ is an irreducible unitary representation of $G$ on a Hilbert space. If the underlying $(\mathfrak{g},K)$-module $\pi_K$ is discretely decomposable as a $(\mathfrak{g}',K')$-module, then $\pi$ decomposes as a Hilbert direct sum of irreducible subrepresentations of $G'$:\[\pi=\widehat{\bigoplus\limits_{\tau}}m(\tau)\tau\]where each $\tau$ is an irreducible representation of $G'$ with the multiplicity $m(\tau)\in\mathbb{Z}_{\geq0}\cup\{+\infty\}$. Moreover, denote by $\tau_{K'}$ the underlying $(\mathfrak{g}',K')$-module of $\tau$, and then one has $m(\tau)=\mathrm{Hom}_{\mathfrak{g}',K'}(\tau_{K'},\pi_K)$.

If $\pi$ decomposes as a Hilbert direct sum of irreducible subrepresentations of $G'$ with finite multiplicities, i.e., $m(\tau)\in\mathbb{Z}_{\geq0}$ for all $\tau$, then $\pi$ is called $G'$-admissible.

Retain all the settings as in Section 2.2. Now let $G$ be an exceptional Lie group, and $\Gamma$ a Klein four subgroup of $\mathrm{Aut}G$, which defines a Klein four symmetric pair $(G,G^\Gamma)$. In this section, the author conjectures a ``upper bound" of Klein four symmetric pair $(G,G^\Gamma)$ such that there exists a nontrivial irreducible unitary representation of $G$ which is $G^\Gamma$-admissible. According to \propositionref{6}, if $(G,G^\Gamma)$ is a Klein four symmetric pair of holomorphic type, any highest weight representation of $G$ is $K^\Gamma$-admissible, and hence is $G^\Gamma$-admissible by \cite[Theorem 1.2]{Ko2}. Now the author just assumes that $(G,G^\Gamma)$ is not of holomorphic type. Further, assume that $G^\Gamma$ is noncompact. Suppose that an irreducible unitary representation $\pi$ of $G$ is $G^\Gamma$-admissible, then it follows from \cite[Theorem 1.2]{Ko2} that $\pi$ is $G^\sigma$-admissible for all $\sigma\in\Gamma$. The following conjecture was raised by Toshiyuki KOBAYASHI in \cite[Conjecture E]{Ko1}.
\begin{conjecture}\label{17}
Suppose that $(G,G')$ is a symmetric pair. Then $\pi_K$ is discretely decomposable as a $(\mathfrak{g}',K')$-module if and only if $\pi$ is $G'$-admissible.
\end{conjecture}
If \conjectureref{17} is true, $\pi_K$ is discretely decomposable as a $(\mathfrak{g}^\sigma,K^\sigma)$-module for all $\sigma\in\Gamma$. By the classification of discretely decomposable $(\mathfrak{g},K)$-modules in \cite[Theorem 5.2 \& Table 1]{KO2}, the only case is that $G=\mathrm{E}_{6(-14)}$ and $G^\sigma=\mathrm{F}_{4(-20)}$ for some $\sigma\in\Gamma$. Recall the notations in Section 4, Klein four subgroups of $\mathrm{Aut}\mathfrak{e}_{6(-14)}$ correspond to elementary abelian 2-subgroups of rank 3 of $\mathrm{Aut}\mathfrak{e}_{6(-78)}$, and $\mathrm{Aut}\mathfrak{e}_{6(-14)}$ is a noncompact dual of $\mathrm{Aut}\mathfrak{e}_{6(-78)}$ defined by an involutive automorphism conjugate to $\sigma_2$. Moreover, if $G=\mathrm{E}_{6(-14)}$ and $G^\sigma=\mathrm{F}_{4(-20)}$, then $\sigma$ must correspond to an involutive automorphism conjugate to $\sigma_3$ by Section 4.1 and \cite[Table 1]{HY}. Therefore, one needs to find all the elementary abelian 2-subgroups of rank 3 of $\mathrm{Aut}\mathfrak{e}_{6(-78)}$ which contains two elements conjugate to $\sigma_2$ and $\sigma_3$ respectively. Recall the subgroups $F_{r,s}$ in Section 4.1, by \cite[Lemma 8]{H1} and \cite[Theorem 6.3 \& Table 3]{Y}, there is only one such subgroups: $\langle x_0,x_4,x_5\rangle$.

In the subgroup $\langle x_0,x_4,x_5\rangle$, the elements conjugate to $\sigma_2$ are exactly $x_4$, $x_5$, and $x_4x_5$ by \cite[Lemma 8]{H1}. By \cite{HY}, $x_0$, $x_0x_4$, $x_0x_5$, and $x_0x_4x_5$ are all conjugate to $\sigma_3$. By \cite[Proposition 6.9]{Y} and a quick computation as the proof for \cite[Lemma 12]{H1}, one only needs to consider the case that $\mathfrak{e}_{6(-14)}\cong\mathfrak{e}_{6(-78)}^{x_5}+\sqrt{-1}\mathfrak{e}_{6(-78)}^{-x_5}$ and $\Gamma=\langle x_0,x_4\rangle$. But this just gives the Klein four symmetric pair $(\mathfrak{e}_{6(-14)},\mathfrak{so}(8,1))$.

Thus, based on \conjectureref{17}, the author conjectures a ``upper bound" for $G^\Gamma$-admissible nontrivial irreducible unitary representations of $G$, taking compact $G^\Gamma$ into consideration.
\begin{conjecture}\label{18}
Let $G$ be an exceptional Lie group of Hermitian type, i.e., $G=\mathrm{E}_{6(-14)}$ or $\mathrm{E}_{7(-25)}$, and $(G,G^\Gamma)$ be a Klein four symmetric pair. Denote by $\mathfrak{g}$ and $\mathfrak{g}^\Gamma$ the complexified Lie algebras of $G$ and $G^\Gamma$ respectively. If there exists a nontrivial unitary irreducible representation of $G$ that is $G^\Gamma$-admissible, then $(G,G^\Gamma)$ satisfies one of the following conditions:
\begin{enumerate}[(i)]
\item $G^\Gamma$ is compact.
\item $(G,G^\Gamma)$ is a Klein for symmetric pair of holomorphic type.
\item $(\mathfrak{g},\mathfrak{g}^\Gamma)=(\mathfrak{e}_{6(-14)},\mathfrak{so}(8,1))$.
\end{enumerate}
\end{conjecture}
\begin{remark}\label{19}
\conjectureref{17} implies \conjectureref{18}. Even if \conjectureref{18} is true, it is not known whether there does exist a nontrivial unitary irreducible representation of $G$ that is $G^\Gamma$-admissible for condition (i) or (iii).
\end{remark}

\end{document}